\documentclass[leqno,12pt]{article} 
\usepackage{latexsym,rotate,eucal,cite}
\usepackage{amsmath,amsthm,amssymb,amsxtra}
\usepackage{amsthm}
\usepackage{cite}
\numberwithin{equation}{section}
\theoremstyle{plain} 
\newtheorem{theorem}{\indent\sc Theorem}[section]
\newtheorem{lemma}[theorem]{\indent\sc Lemma}
\newtheorem{corollary}[theorem]{\indent\sc Corollary}

\theoremstyle{definition} 
\newtheorem{definition}[theorem]{\indent\sc Definition}
\newtheorem{remark}[theorem]{\indent\sc Remark}

\makeatletter
\def\address#1#2{\begingroup
\noindent\parbox[t]{7.8cm}{%
\small{\scshape\ignorespaces#1}\par\vskip1ex
\noindent\small{\itshape E-mail address}%
\/: #2\par\vskip4ex}\hfill%
\endgroup}%
\makeatother
\title{\uppercase{Biharmonic hypersurfaces with three distinct principal curvatures in Euclidean space}} 
\author{
\bigskip \\
\textsc{Yu Fu} 
}
\date{} 
\begin{document}

\maketitle

\footnote{ 
$^{*}$Thanks. }
\begin{abstract}
The well known Chen's conjecture on biharmonic submanifolds states
that a biharmonic submanifold in a Euclidean space is a minimal one
([10-13, 16, 18-21, 8]). For the case of hypersurfaces, we know that
Chen's conjecture is true for biharmonic surfaces in $\mathbb E^3$
([10], [24]), biharmonic hypersurfaces in $\mathbb E^4$ ([23]), and
biharmonic hypersurfaces in $\mathbb E^m$ with at most two distinct
principal curvature  ([21]). The most recent work of Chen-Munteanu
[18] shows that Chen's conjecture is true for $\delta(2)$-ideal
hypersurfaces in $\mathbb E^m$, where a $\delta(2)$-ideal
hypersurface is a hypersurface whose principal curvatures take three
special values: $\lambda_1, \lambda_2$ and $\lambda_1+\lambda_2$. In
this paper, we prove that Chen's conjecture is true for
hypersurfaces with three distinct principal curvatures in $\mathbb
E^m$ with arbitrary dimension, thus, extend all the above-mentioned
results. As an application we also show that Chen's conjecture is
true for $O(p)\times O(q)$-invariant hypersurfaces in Euclidean
space $\mathbb E^{p+q}$.
\end{abstract}
\section{Introduction}
\hspace*{\parindent}
Investigating the properties of biharmonic submanifolds in Euclidean
spaces was initiated by B. Y. Chen in the middle of 1980s in his
study on finite type submanifolds. At first, B.Y. Chen proved that
biharmonic surfaces in Euclidean 3-spaces are minimal, which was
also independently proved by G. Y. Jiang \cite{jiang1986}. Later on,
I. Dimitri\'{c} in his doctoral thesis \cite{Dimitri1989} and his
paper \cite{Dimitri1992} proved that any biharmonic curve in
Euclidean spaces $\mathbb E^m$ is a part of a straight line (i.e.
minimal); any biharmonic submanifold of finite type in $\mathbb E^m$
is minimal; any pseudo umbilical submanifold $M^n$ in $\mathbb E^m$
with $n\neq4$ is minimal, and any biharmonic hypersurface in
$\mathbb E^m$ with at most two distinct principal curvatures is
minimal. Hence, based on these results B. Y. Chen \cite{Chen1991} in
1991 made the following well known conjecture:

{\em Every biharmonic submanifold of Euclidean spaces is minimal}.

In 1995, the conjecture was proved by T. Hasanis and T. Vlachos
\cite{HasanisVlachos1995} for hypersurfaces in Euclidean 4-spaces,
(see also Defever's work\cite{defever1998} with a different proof).
However, the conjecture remains open. The main difficulty is that
the conjecture is a local problem and how to understand the local
structure of submanifolds satisfying $\Delta\overrightarrow{H}=0$.
Nevertheless, the study of the conjecture is quite active nowadays.
Recently, B. Y. Chen and M. I. Munteanu \cite{chenMunteanu2013}
proved that Chen's conjecture is true for $\delta(2)$-ideal and
$\delta(3)$-ideal hypersurfaces of a Euclidean space with arbitrary
dimension, where the principal curvatures of such hypersurfaces
takes special values. Under the assumption of completeness, K.
Akutagawa and S. Maeta \cite{akutagawa2013} proved that biharmonic
properly immersed submanifolds in Euclidean spaces are minimal.

On the other hand, from the view of $k$-harmonic maps, one can
define a biharmonic map between Riemannian manifolds if it is a
critical point of the bienergy functional. G. Y. Jiang in
\cite{jiang1986} showed that a smooth map is biharmonic if and only
if its bitension field vanishes identically. In the past ten years,
there exists a lot of remarkable work on biharmonic submanifolds in
spheres or even in generic Riemannian manifolds (see, for instance
[3, 5-9, 26-28]). Nowadays, investigating the properties of
biharmonic submanifolds is becoming a very active field of study.

In contrast to the submanifolds in Euclidean spaces, Chen's
conjecture is not always true for submanifolds in pseudo-Euclidean
spaces. This fact was achieved by B. Y. Chen and S. Ishikawa
\cite{chenishika1991,chenishika1999} who constructed several
examples of proper biharmonic surfaces in 4-dimensional
pseduo-Euclidean spaces $\mathbb E_s^4$ ($s=1, 2, 3$). But for
hypersurfaces in pseudo-Euclidean spaces, B. Y. Chen and S. Ishikawa
proved in \cite{chenishika1991,chenishika1999} that biharmonic
surfaces in pseudo-Euclidean 3-spaces are minimal, and A.
Arvanitoyeorgosa et al. \cite{ADKP} proved that biharmonic
Lorentzian hypersurfaces in Minkowski 4-spaces are minimal.

As we known, Chen's conjecture for hypersurfaces in $\mathbb E^4$
and for hypersurfaces in $\mathbb E^m$ with two distinct principal
curvatures were solved by Hasanis-Vlachos and Dimitri\'{c},
respectively. It is natural to study biharmonic hypersurfaces with
three distinct principal curvatures as the next step. Following B.
Y. Chen, I. Dimitri\'{c}, F. Defever et. al's techniques, we make
further progress on the conjecture. In a previous work \cite{fuyu},
we proved that Chen's conjecture is true for biharmonic
hypersurfaces with three distinct principal curvatures in $\mathbb
E^5$. In this paper, we are able to solve that general case.
Precisely, we will prove that biharmonic hypersurfaces with at most
three distinct principal curvatures in $\mathbb E^{n+1}$ with
arbitrary dimension are minimal. As an immediate conclusion, we show
that biharmonic $O(p)\times O(q)$-invariant hypersurfaces in
Euclidean spaces $\mathbb E^{p+q}$ are minimal.

\section{Preliminaries}
Let $x: M^n\rightarrow\mathbb{E}^{n+1}$ be an isometric immersion of
a hypersurface $M^n$ into $\mathbb{E}^{n+1}$. Denote the Levi-Civita
connections of $M^n$ and $\mathbb{E}^m$ by $\nabla$ and
$\tilde\nabla$, respectively. Let $X$ and $Y$ denote vector fields
tangent to $M^n$ and let $\xi$ be a unite normal vector field. Then
the Gauss and Weingarten formulas are given, respectively, by (cf.
\cite{chenbook2011, chenbook1973, chenbook1984})
\begin{eqnarray}
\tilde\nabla_XY&=&\nabla_XY+h(X,Y),\label{l23}\\
\tilde\nabla_X\xi&=&-AX,\label{l16}
\end{eqnarray}
where $h$ is the second fundamental form, and $A$ is the shape
operator. It is well known that the second fundamental form $h$ and
the shape operator $A$ are related by
\begin{eqnarray}\label{l3}
\langle h(X,Y),\xi\rangle=\langle AX,Y\rangle.
\end{eqnarray}
The mean curvature vector field $\overrightarrow{H}$ is given by
\begin{eqnarray}
\overrightarrow{H}=\frac{1}{n}{\rm trace}~h.
\end{eqnarray}
The Gauss and Codazzi equations are given, respectively, by
\begin{eqnarray*}
R(X,Y)Z=\langle AY,Z\rangle AX-\langle AX,Z\rangle AY,
\end{eqnarray*}
\begin{eqnarray*}
(\nabla_{X} A)Y=(\nabla_{Y} A)X,
\end{eqnarray*}
where $R$ is the curvature tensor and $(\nabla_XA)Y$ is defined by
\begin{eqnarray}\label{l7}
(\nabla_XA)Y=\nabla_X(AY)-A(\nabla_XY)
\end{eqnarray}
for all $X, Y, Z$ tangent to $M$.

Let $\Delta$ be the Laplacian operator of a submanifold $M$. For an
isometric immersion $ x: M^n\rightarrow\mathbb{E}^{m}$, the mean
curvature vector field $\overrightarrow{H}$ in $\mathbb{E}^{m}$
satisfies (see, for instance \cite{chenbook2011}, p. 44)
\begin{eqnarray*}
{\Delta}x=-n\overrightarrow{H}.
\end{eqnarray*}
\begin{definition}
Let $x:M^n\rightarrow\mathbb{E}^{m}$ be an isometric immersion of a
Riemannian $n$-manifold $M$ into a Euclidean space $\mathbb{E}^{m}$.
Then $M^n$ is called a biharmonic submanifold in $\mathbb E^m$ if
and only if $\Delta \overrightarrow{H}=0$, or equivalently,
$\Delta^2 x=0$.
\end{definition}
By Definition 2.1, it is clear that any minimal submanifolds in a
Euclidean space $\mathbb{E}^{m}$ must be trivially biharmonic. A
biharmonic submanifold in a Euclidean space $\mathbb{E}^{m}$ is
called proper biharmonic if it is not minimal.

Let $M^n$ be a hypersurface in $\mathbb E^{n+1}$. Assume that
$\overrightarrow{H}=H\xi$. Note that $H$ denotes the mean curvature.
By identifying the normal and the tangent parts of the biharmonic
condition $\Delta \overrightarrow{H}=0$, we obtain necessary and
sufficient conditions for $M^n$ to be biharmonic in $\mathbb
E^{n+1}$, namely
\begin{eqnarray}
&&\Delta H+H {\rm trace}\, A^2 =0,\\
&&2A\,{\rm grad}H+n\, H{\rm grad}H=0,
\end{eqnarray}
where the Laplace operator $\Delta$ acting on scalar-valued function
$f$ is given by (e.g., [13])
\begin{eqnarray}
\Delta f=-\sum_{i=1}^n(e_ie_i f-\nabla_{e_i}e_i f).
\end{eqnarray}
Here, $\{e_1,\ldots,e_n\}$ is a local orthonormal tangent frame on
$M^n$.

\section{Biharmonic hypersurfaces with three distinct principal curvatures in $\mathbb E^{n+1}$}
From now on, we concentrate on biharmonic hypersurfaces $M^n$ in a
Euclidean space $\mathbb E^{n+1}$ with $n\geq4$.

Assume that the mean curvature $H$ is not constant.

Observe from (2.7) that ${\rm grad}\,H$ is an eigenvector of the
shape operator $A$ with the corresponding principal curvature
$-\frac{n}{2}H$. Without loss of generality, we can choose $e_1$
such that $e_1$ is parallel to ${\rm grad}\,H$, and therefore the
shape operator $A$ of $M^n$ takes the following form with respect to
a suitable orthonormal frame $\{e_1,\ldots, e_n\}$.
\begin{eqnarray}
A=\left( \begin{array}{cccc} \lambda_1&\\
&\lambda_2 &\\&&\ddots &\\&&&\lambda_n
\end{array} \right),
\end{eqnarray}
where $\lambda_i$ are the principal curvatures and
$\lambda_1=-\frac{n}{2}H$. Let us express ${\rm grad}H$ as
\begin{eqnarray*}
{\rm grad}H=\sum_{i=1}^ne_i(H)e_i.
\end{eqnarray*}
Since $e_1$ is parallel to ${\rm grad}\,H$, it follows that
\begin{eqnarray}
e_1(H)\neq0,\quad e_i(H)=0, \quad i=2, 3, \ldots, n.
\end{eqnarray}
We write
\begin{eqnarray}
\nabla_{e_i}e_j=\sum_{k=1}^n\omega_{ij}^ke_k,\quad i,j=1, 2, \ldots,
n.
\end{eqnarray}
The compatibility conditions $\nabla_{e_k}\langle e_i,e_i\rangle=0$
and $\nabla_{e_k}\langle e_i,e_j\rangle=0$ imply respectively that
\begin{eqnarray}
\omega_{ki}^i=0,\quad \omega_{ki}^j+\omega_{kj}^i=0,
\end{eqnarray}
for $i\neq j$ and $i, j, k=1, 2, \ldots, n$. Furthermore, it follows
from (3.1) and (3.3) that the Codazzi equation yields
\begin{eqnarray}
e_i(\lambda_j)=(\lambda_i-\lambda_j)\omega_{ji}^j,\\
(\lambda_i-\lambda_j)\omega_{ki}^j=(\lambda_k-\lambda_j)\omega_{ik}^j
\end{eqnarray}
for distinct $i, j, k=1, 2, \ldots, n$.

Since $\lambda_1=-\frac{n}{2}H$, from (3.2) we get
\begin{eqnarray*}
[e_i,e_j](\lambda_1)=0,\quad i, j=2, 3, \ldots, n, \quad i\neq j,
\end{eqnarray*}
which yields directly
\begin{eqnarray}
\omega_{ij}^1=\omega_{ji}^1,
\end{eqnarray}
for distinct $i, j=2, 3, \ldots, n$.

Now we show that $\lambda_j\neq\lambda_1$ for $j=2, 3, \ldots, n$.
In fact, if $\lambda_j=\lambda_1$ for $j\neq1$, by putting $i=1$ in
(3.5) we have that
\begin{eqnarray}
0=(\lambda_1-\lambda_j)\omega_{j1}^j=e_1(\lambda_j)=e_1(\lambda_1),
\end{eqnarray}
which contradicts the first expression of (3.2).

By the assumption, $M^n$ has three distinct principal curvatures.
Without loss of generality, we assume that
\begin{eqnarray}
&&\lambda_2=\lambda_3=\cdots=\lambda_p=\alpha,\nonumber\\
&&\lambda_{p+1}=\lambda_{p+2}=\cdots=\lambda_n=\beta, \nonumber\\
&&\frac{n+1}{2}\leq p<n.
\end{eqnarray}
By the definition (2.4) of $\overrightarrow{H}$, we have
$nH=\sum_{i=1}^n\lambda_i$. Hence
\begin{eqnarray}
\beta=\frac{\frac{3}{2}nH-(p-1)\alpha}{n-p}.
\end{eqnarray}
Since $\lambda_j\neq\lambda_1$ for $i=2,\ldots,n$, we obtain
\begin{eqnarray}
\alpha\neq-\frac{n}{2}H, ~~\frac{3n}{2(n-1)}H,~~
\frac{n^2-(p-3)n}{2(p-1)}H.
\end{eqnarray}
The multiplicities of principal curvatures $\alpha$ and $\beta$ are
$p-1$ and $n-p$, respectively.

In the following, we will state a key conclusion for later use.
\begin{lemma}
Let $M^n$ be a proper biharmonic hypersurface with three distinct
principal curvatures in $\mathbb E^{n+1}$. Then $e_i(\lambda_j)=0$
for $i=1,2,\ldots,n$ and $j=2,3,\ldots,n$.
\end{lemma}
\begin{proof}
Consider the equation (3.5). Since $n\geq4$, it follows from (3.9)
that $p-1\geq2$. For $i, j=2, 3,\ldots,p$ and $i\neq j$ in (3.5),
one has
\begin{eqnarray}
e_i(\alpha)=0, \quad i=2, 3, \ldots, p.
\end{eqnarray}
If the multiplicity of principal curvature $\beta$ satisfies
$n-p\geq2$, then for $i, j=p+1, \ldots,n$ and $i\neq j$ in (3.5) we
have
\begin{eqnarray}
e_i(\beta)=0, \quad i=p+1, \ldots, n.
\end{eqnarray}
Hence, the conclusion follows directly from (3.2), (3.10), (3.12)
and (3.13).

If the multiplicity of principal curvature $\beta$ is one, namely
$p=n-1$, then from (3.12) we only need to show that $e_n(\alpha)=0$.

Let us compute
$[e_1,e_i](H)=\big(\nabla_{e_1}e_i-\nabla_{e_i}e_1\big)(H)$ for
$i=2, \ldots, n$. From the first expression of (3.4), we have
$\omega_{i1}^1=0$. For $j=1$ and $i\neq1$ in (3.5), by (3.2) we have
$\omega_{1i}^1=0$ $(i\neq1)$. Hence we have
\begin{eqnarray}
e_ie_1(H)=0,\quad i=2, \ldots, n.
\end{eqnarray}
By (3.12), a similar way can also show that
\begin{eqnarray}
e_ie_1(\alpha)=0,\quad i=2, \ldots, n-1.
\end{eqnarray}

 For $j=1$, $k, i\neq1$ in (3.6) we have
\begin{eqnarray*}
(\lambda_i-\lambda_1)\omega_{ki}^1=(\lambda_k-\lambda_1)\omega_{ik}^1,
\end{eqnarray*}
which together with (3.7) yields
\begin{eqnarray}
\omega_{ij}^1=0, \quad i\neq j,\quad i, j=2,\ldots, n.
\end{eqnarray}
Moreover, combining (3.16) with the second equation of (3.4) gives
\begin{eqnarray}
\omega_{i1}^j=0, \quad i, j=2,\ldots, n, \quad j\neq i.
\end{eqnarray}
It follows from (3.5) that
\begin{eqnarray}
\omega_{i1}^i=\frac{e_1(\lambda_i)}{\lambda_1-\lambda_i}, \quad
i=2,\ldots, n.
\end{eqnarray}
For $k=2$ and $i=n$ in (3.6), we have
\begin{eqnarray*}
(\lambda_n-\lambda_j)\omega_{2n}^j=(\lambda_2-\lambda_j)\omega_{n2}^j,
\end{eqnarray*}
which yields
\begin{eqnarray*}
\omega_{2n}^j=0, \quad j=3,\ldots, n-1.
\end{eqnarray*}
Hence, from the first expression of (3.4) and (3.16) we get
\begin{eqnarray}
\omega_{2n}^j=0, \quad j=1, 3,\ldots, n.
\end{eqnarray}
Also, (3.5) yields
\begin{eqnarray}
\omega_{2n}^2=\frac{e_n(\alpha)}{\lambda_n-\alpha}.
\end{eqnarray}

From the Gauss equation and (3.1) we have $R(e_2,e_n)e_1=0$. Recall
the definition of Gauss curvature tensor
\begin{eqnarray}
R(X,Y)Z=\nabla_X\nabla_YZ-\nabla_Y\nabla_XZ-\nabla_{[X,Y]}Z.
\end{eqnarray}
It follows from (3.15), (3.17-20) and (3.4) that
\begin{eqnarray*}
&&\nabla_{e_2}\nabla_{e_n}e_1=\frac{e_1(\lambda_n)e_n(\alpha)}{(\lambda_1-\lambda_n)(\lambda_n-\alpha)}e_2,\\
&&\nabla_{e_n}\nabla_{e_2}e_1=e_n(\frac{e_1(\alpha)}{\lambda_1-\alpha})e_2+\frac{e_1(\alpha)}{\lambda_1-\alpha}\sum_{k=3}^{n}\omega_{n2}^ke_k,\\
&&\nabla_{[e_2,e_n]}e_1=\frac{e_n(\alpha)e_1(\alpha)}{(\lambda_n-\alpha)(\lambda_1-\alpha)}
e_2-\frac{e_1(\alpha)}{\lambda_1-\alpha}\sum_{k=3}^{n}\omega_{n2}^ke_k.
\end{eqnarray*}
Hence we obtain
\begin{eqnarray}
e_n(\frac{e_1(\alpha)}{\lambda_1-\alpha})=\Big(\frac{e_1(\lambda_n)}{\lambda_1-\lambda_n}-
\frac{e_1(\alpha)}{\lambda_1-\alpha}\Big)\frac{e_n(\alpha)}{\lambda_n-\alpha}.
\end{eqnarray}
Note that $\lambda_1=-\frac{n}{2}H$ and
$\lambda_n=\beta=\frac{3}{2}nH-(n-2)\alpha$ in this case.

Equation (3.22) can be rewritten as
\begin{eqnarray}
e_ne_1(\alpha)=\Big\{-\frac{e_1(\alpha)}{\lambda_1-\alpha}+\Big(\frac{e_1(\lambda_n)}{\lambda_1-\lambda_n}-
\frac{e_1(\alpha)}{\lambda_1-\alpha}\Big)\frac{\lambda_1-\alpha}{\lambda_n-\alpha}\Big\}e_n(\alpha).
\end{eqnarray}
By (3.23), we compute
\begin{eqnarray}
e_n(\frac{e_1(\lambda_n)}{\lambda_1-\lambda_n})&=&-(n-2)\Big(\frac{e_ne_1(\alpha)}{\lambda_1-\lambda_n}
+\frac{e_1(\lambda_n)e_n(\alpha)}{(\lambda_1-\lambda_n)^2}\Big)\nonumber\\
&=&-(n-2)\frac{e_n(\alpha)}{\lambda_1-\lambda_n}\Big(\frac{e_1(\lambda_n)}{\lambda_1-\lambda_n}-
\frac{e_1(\alpha)}{\lambda_1-\alpha}\Big)\frac{\lambda_1+\lambda_n-2\alpha}{\lambda_n-\alpha}.
\end{eqnarray}

It follows from (3.5) and the second expression of (3.4) that
\begin{eqnarray}
\omega_{ii}^1=-\omega_{i1}^i=-\frac{e_1(\lambda_i)}{\lambda_1-\lambda_i}.
\end{eqnarray}
Now consider the equation (2.6). It follows from (2.8), (3.1) and
(3.25) that
\begin{equation}
-e_1e_1(H)-\Big(\frac{(n-2)e_1(\alpha)}{\lambda_1-\alpha}+\frac{e_1(\lambda_n)}{\lambda_1-\lambda_n}\Big)e_1(H)+H[{\lambda_1}^2+(n-2)\alpha^2+{\lambda_n}^2]=0.
\end{equation}
Differentiating (3.26) along $e_n$, by (3.22) and (3.24) we get
\begin{equation*}
\Big\{\frac{2}{\lambda_1-\lambda_n}\Big(\frac{e_1(\lambda_n)}{\lambda_1-\lambda_n}-
\frac{e_1(\alpha)}{\lambda_1-\alpha}\Big)e_1(H)+H\big(-3nH+2(n-1)\alpha\big)\Big\}e_n(\alpha)=0.
\end{equation*}
If $e_n(\alpha)\neq0$, then the above equation becomes
\begin{equation}
\frac{2}{\lambda_1-\lambda_n}\Big(\frac{e_1(\lambda_n)}{\lambda_1-\lambda_n}-\frac{e_1(\alpha)}{\lambda_1-\alpha}\Big)e_1(H)+H\big(-3nH+2(n-1)\alpha\big)=0.
\end{equation}
Differentiating (3.27) along $e_n$ and using (3.22) and (3.24)
again, one has
\begin{eqnarray}
&&\frac{2n(4-n)H+2(n-2)(n-1)\alpha}{(\lambda_1-\lambda_n)(\lambda_n-\alpha)}\Big(\frac{e_1(\lambda_n)}{\lambda_1-\lambda_n}-\frac{e_1(\alpha)}{\lambda_1-\alpha}\Big)e_1(H)\nonumber
\\&&+H\big((-7n+10)nH+4(n-1)(n-2)\alpha\big)=0.
\end{eqnarray}
Therefore, combining (3.28) with (3.27) gives
\begin{eqnarray*}
3(n-2)H[3nH-2(n-1)\alpha]^2=0,
\end{eqnarray*}
which implies that
\begin{eqnarray*}
\alpha=\frac{3n}{2(n-1)}H.
\end{eqnarray*}
This contradicts to (3.11). Hence, we obtain $e_n(\alpha)=0$, which
completes the proof of Lemma 3.1.
\end{proof}

Now, we are ready to express the connection coefficients of
hypersurfaces.

For $j=1$ and $i=2,\ldots,n$ in (3.5), by (3.2) we get
$\omega_{1i}^1=0$. Moreover, by the first and second expressions of
(3.4) we have
\begin{eqnarray}
\omega_{1i}^1=\omega_{11}^i=0,\quad i=1, \ldots, n.
\end{eqnarray}
For $i=1$, $j=2, \ldots, n$ in (3.5), we obtain
\begin{eqnarray}
\omega_{j1}^j=-\omega_{jj}^1=\frac{e_1(\lambda_j)}{\lambda_1-\lambda_j},
\quad j=2, \ldots, n.
\end{eqnarray}
For $i=p+1,\ldots,n$, $j=2,\ldots, p$ in (3.5), by (3.2) we have
\begin{eqnarray}
\omega_{ji}^j=-\omega_{jj}^i=0.
\end{eqnarray}
Similarly, for $i=2,\ldots, p$, $j=p+1,\ldots,n$ in (3.5), we also
have
\begin{eqnarray}
\omega_{ji}^j=-\omega_{jj}^i=0.
\end{eqnarray}
For $i=1$, by choosing $j, k=2,\ldots,p$ or $k, j=p+1,\ldots,n$
($j\neq k$) in (3.6), we have
\begin{eqnarray}
\omega_{k1}^j=\omega_{kj}^1=0.
\end{eqnarray}
For $i=2,\ldots, p$ and $j, k=p+1,\ldots, n$ ($j\neq k$) in (3.6),
we get
\begin{eqnarray}
\omega_{ki}^j=\omega_{kj}^i=0.
\end{eqnarray}
For $i=2,\ldots, p$, $j=1$ and $k=p+1,\ldots, n$ in (3.6), we have
\begin{eqnarray*}
(\alpha-\lambda_1)\omega_{ki}^1=(\beta-\lambda_1)\omega_{ik}^1,
\end{eqnarray*}
which together with (3.7) and the second expression of (3.4) gives
\begin{eqnarray}
\omega_{ki}^1=\omega_{ik}^1=\omega_{k1}^i=\omega_{i1}^k=0.
\end{eqnarray}
For $i=2,\ldots, p$, $k=1$ and $j=p+1,\ldots, n$ in (3.6), we obtain
\begin{eqnarray*}
(\beta-\alpha)\omega_{1i}^j=(\lambda_1-\alpha)\omega_{i1}^j,
\end{eqnarray*}
which together with (3.35) yields
\begin{eqnarray}
\omega_{1i}^j=\omega_{1j}^i=0.
\end{eqnarray}

Combining (3.29-3.36) with (3.4) and summarizing, we have the
following lemma.
\begin{lemma}
Let $M^n$ be a biharmonic hypersurface with non-constant mean
curvature in Euclidean space $\mathbb E^{n+1}$, whose shape operator
given by (3.1) with respect to an orthonormal frame $\{e_1, \ldots,
e_n\}$. Then we have
\begin{eqnarray*}
&&\nabla_{e_1}e_1=0;~
\nabla_{e_i}e_1=\frac{e_1(\lambda_i)}{\lambda_1-\lambda_i}e_i,~i=2,\ldots,n;\\
&&\nabla_{e_i}e_j=\sum_{k=2, k\neq j}^{p}\omega_{ij}^ke_k,
~i=1,\ldots,n,~j=2,\ldots,p,~i\neq j;\\
&&\nabla_{e_i}e_i=-\frac{e_1(\lambda_i)}{\lambda_1-\lambda_i}e_1+\sum_{k=2,
k\neq i}^{p}\omega_{ii}^ke_k,
~i=2,\ldots,p;\\
&&\nabla_{e_i}e_j=\sum_{k=p+1, k\neq j}^{n}\omega_{ij}^ke_k,
~i=1,\ldots,n,~j=p+1,\ldots,n,~i\neq j;\\
&&\nabla_{e_i}e_i=-\frac{e_1(\lambda_i)}{\lambda_1-\lambda_i}e_1+\sum_{k=p+1,
k\neq i}^{n}\omega_{ii}^ke_k, ~i=p+1,\ldots,n,
\end{eqnarray*}
where $\omega_{ki}^j=-\omega_{kj}^i$ for $i\neq j$ and $i, j,
k=1,\ldots, n$.
\end{lemma}
Let us introduce two smooth functions $A$ and $B$ as follows
\begin{eqnarray}
A=\frac{e_1(\alpha)}{\lambda_1-\alpha},\quad
B=\frac{e_1(\beta)}{\lambda_1-\beta}.
\end{eqnarray}
One can compute the curvature tensor $R$ by Lemma 3.2 and apply the
Gauss equation for different values of $X$, $Y$ and $Z$. After
comparing the coefficients with respect to the orthonormal basis
$\{e_1, \ldots, e_n\}$ we get the following:
\begin{itemize}
\item $X=e_1, Y=e_2, Z=e_1$,
\begin{eqnarray}
e_1(A)+A^2=-\lambda_1\alpha;
\end{eqnarray}
\item $X=e_1, Y=e_n, Z=e_1$,
\begin{eqnarray}
e_1(B)+B^2=-\lambda_1\beta;
\end{eqnarray}
\item $X=e_n, Y=e_2, Z=e_n$,
\begin{eqnarray}
AB=-\alpha\beta.
\end{eqnarray}
\end{itemize}
Consider the  equation (2.6) again. It follows from (2.8), (3.1),
(3.37) and Lemma 3.2 that
\begin{equation}
-e_1e_1(H)-\big[(p-1)A+(n-p)B\big]e_1(H)+H\big[\lambda^2_1+(p-1)\alpha^2+(n-p)\beta^2\big]=0.
\end{equation}
We will derive a key equation for later use.
\begin{lemma}
The functions $A$ and $B$ are related by
\begin{eqnarray}
\big[(4-p)A+(3+p-n)B\big]e_1(H)+\frac{3n^2(n+6-p)}{4(n-p)}H^3\nonumber\\
-\frac{3n(n-2+4p)}{2(n-p)}H^2\alpha+\frac{3n(p-1)}{n-p}H\alpha^2=0.
\end{eqnarray}
\end{lemma}
\begin{proof}
By (3.37), equations (3.38) and (3.39) further reduce to
\begin{eqnarray}
e_1e_1(\alpha)+2Ae_1(\alpha)-Ae_1(\lambda_1)+\lambda_1\alpha(\lambda_1-\alpha)=0,\\
e_1e_1(\beta)+2Be_1(\beta)-Be_1(\lambda_1)+\lambda_1\beta(\lambda_1-\beta)=0.
\end{eqnarray}
Since $\alpha$ and $\beta$ are related by
$(n-p)\beta+(p-1)\alpha=\frac{3nH}{2}$, it follows from (3.37) that
\begin{eqnarray}
e_1(\alpha)=\frac{3n}{2(p-1)}e_1(H)-\frac{n-p}{p-1}B(\lambda_1-\beta),\\
e_1(\beta)=\frac{3n}{2(n-p)}e_1(H)-\frac{p-1}{n-p}A(\lambda_1-\alpha).
\end{eqnarray}
Substituting (3.45) and (3.46) into (3.47) and (3.48), respectively,
by (3.40) we have
\begin{eqnarray}
e_1e_1(\alpha)+\big(\frac{3}{p-1}+\frac{1}{2}\big)nAe_1(H)+\frac{2(n-p)}{p-1}(\lambda_1-\beta)\alpha\beta+\lambda_1\alpha(\lambda_1-\alpha)=0,\\
e_1e_1(\beta)+\big(\frac{3}{n-p}+\frac{1}{2}\big)nBe_1(H)+\frac{2(p-1)}{n-p}(\lambda_1-\alpha)\alpha\beta+\lambda_1\beta(\lambda_1-\beta)=0,
\end{eqnarray}
where we use $\lambda_1=-\frac{nH}{2}$. By using
$(n-p)\beta+(p-1)\alpha=\frac{3nH}{2}$, we could eliminate
$e_1e_1(H)$, $e_1e_1(\alpha)$ and $e_1e_1(\beta)$ from (3.41),
(3.47) and (3.48). Consequently, we obtain the desired equation
(3.42).
\end{proof}

Moreover, by using $(n-p)\beta+(p-1)\alpha=\frac{3nH}{2}$ and (3.37)
we have
\begin{eqnarray}
e_1(H)=-\Big[\frac{p-1}{3}H+\frac{2(p-1)}{3n}\alpha\Big]A+\Big[-\frac{n+3-p}{3}H+\frac{2(p-1)}{3n}\alpha\Big]B.
\end{eqnarray}
Substituting (3.49) into (3.42) and using (3.40), we get
\begin{eqnarray}
&&(4-p)(p-1)(nH+2\alpha)A^2+(3+p-n)[n(n+3-p)H-2(p-1)\alpha]B^2\nonumber\\
&&=\frac{n(p-1)(-2p^2+2pn+11p+n-12)}{n-p}H\alpha^2-\frac{2(p-1)^2(2p-n-1)}{n-p}\alpha^3\nonumber\\
&&+\frac{9n^3(n+6-p)}{4(n-p)}H^3+\frac{3n^2(p-1)(2p-2n-15)}{2(n-p)}H^2\alpha.
\end{eqnarray}

Multiplying $A$ and $B$ successively on the equation (3.42), using
(3.40) one gets respectively
\begin{eqnarray}
&&(4-p)A^2e_1(H)-(3+p-n)\alpha\beta e_1(H)\\
&&+\Big[\frac{3n^2(n+6-p)}{4(n-p)}H^3-\frac{3n(n-2+4p)}{2(n-p)}H^2\alpha+\frac{3n(p-1)}{n-p}H\alpha^2\Big]A=0,\nonumber\\
&&(3+p-n)B^2e_1(H)-(4-p)\alpha\beta e_1(H)\\
&&+\Big[\frac{3n^2(n+6-p)}{4(n-p)}H^3-\frac{3n(n-2+4p)}{2(n-p)}H^2\alpha+\frac{3n(p-1)}{n-p}H\alpha^2\Big]B=0.\nonumber
\end{eqnarray}
Differentiating (3.42) along $e_1$, and using (3.38-39) and (3.41)
we have
\begin{eqnarray}
&&\Big[(4-p)(\frac{n}{2}H\alpha-A^2)+(3+p-n)(\frac{n}{2}H\beta-B^2)\Big]e_1(H)\nonumber\\
&&-\Big[(4-p)A+(3+p-n)B\Big]\Big[(p-1)A+(n-p)B\Big]e_1(H)\nonumber\\
&&+\Big[(4-p)A+(3+p-n)B\Big]\Big[\frac{n^2}{4}H^3+(p-1)H\alpha^2+(n-p)H\beta^2\Big]\nonumber\\
&&+\Big[\frac{9n^2(n+6-p)}{4(n-p)}H^2-\frac{3n(n-2+4p)}{n-p}H\alpha+\frac{3n(p-1)}{n-p}\alpha^2\Big]e_1(H)\nonumber\\
&&-\frac{3n(n-2+4p)}{2(n-p)}H^2e_1(\alpha)+\frac{6n(p-1)}{n-p}H\alpha
e_1(\alpha)=0.
\end{eqnarray}
Substituting (3.51), (3.52), (3.42) into (3.53), and using the first
expression of (3.37) we obtain
\begin{eqnarray}
&&\Big[\frac{3n^2(2n-2p+21)}{4(n-p)}H^2-\frac{3n(5p+1)}{n-p}H\alpha+\frac{(p-1)(2n+7)}{n-p}\alpha^2\Big]e_1(H)\\
&&+\Big[\frac{n^2(2pn-2p^2+7n+17p+30)}{4(n-p)}H^3-\frac{3n(3np+2p^2+4p-3n-6)}{2(n-p)}H^2\alpha\nonumber\\
&&+\frac{(p-1)(2np-2n+p-4)}{n-p}H\alpha^2\Big]A+\Big[\frac{n^2\big(2(n-p)^2+15(n-p)+45\big)}{4(n-p)}H^3\nonumber\\
&&-\frac{3n(n^2+np-2p^2+10p+n-8)}{2(n-p)}H^2\alpha
+\frac{(p-1)(2n^2-2np+7n-p-3)}{n-p}H\alpha^2\Big]B=0.\nonumber
\end{eqnarray}
From (3.49), equation (3.54) further reduces to
\begin{eqnarray}
&&\Big[\frac{9}{4}n^3(3n-2p+17)H^3-\frac{3}{2}n^2(-6p^2+11np+43p-11n-37)H^2\alpha\\
&&+n(p-1)(4np-4n+26p+1)H\alpha^2-2(p-1)^2(2n+7)\alpha^3\Big]A\nonumber\\
&&-\Big[\frac{9}{2}(2n-2p+3)H^3+\frac{9}{2}n^2(2p^2+n^2-3np-7p+n-3)H^2\alpha\nonumber\\
&&-2n(p-1)(2n^2-2np+4n-13p-18)H\alpha^2-2(p-1)^2(2n+7)\alpha^3\Big]B=0.\nonumber
\end{eqnarray}
At this moment, we obtain all the desired equations (3.40), (3.50)
and (3.55) concerning $A$ and $B$.

In order to write handily, we introduce several notions: $L, M$
denoting the coefficients of $A$ and $B$ respectively in (3.55), and
$N$ denoting the right-hand side of equal sign in the  equation
(3.50). Then (3.55) and (3.50) become
\begin{eqnarray}
&&LA-MB=0,\\
&&(4-p)(p-1)(nH+2\alpha)A^2\nonumber\\
&&+(3+p-n)[n(n+3-p)H-2(p-1)\alpha]B^2=N.
\end{eqnarray}
Multiplying $LM$ on both sides of the equation (3.57), using (3.56)
and (3.40) we can eliminate both $A$ and $B$. Hence, we have
\begin{eqnarray}
&&(4-p)(p-1)(nH+2\alpha)M^2\alpha\frac{\frac{3}{2}nH-(p-1)\alpha}{n-p}\nonumber\\
&&+(3+p-n)[n(n+3-p)H-2(p-1)\alpha]L^2\alpha\frac{\frac{3}{2}nH-(p-1)\alpha}{n-p}\nonumber\\
&&=LMN.
\end{eqnarray}
In view of (3.58), we notice that the equation should have the form:
\begin{eqnarray}
a_9H^9+a_8H^8\alpha+a_7H^7\alpha^2+a_6H^6\alpha^3+a_5H^5\alpha^4
+a_4H^4\alpha^5+a_3H^3\alpha^6\nonumber\\
+a_2H^2\alpha^7+a_1H\alpha^8+a_0\alpha^9=0
\end{eqnarray}
for constant coefficients $a_i$ concerning $n$ and $p$
($i=0,\ldots,9$). Since $p<n$, from (3.58), (3.55) and (3.50) we
compute $a_9$:
\begin{eqnarray}
a_9=\frac{243n^6(n-p+6)(3n-2p+17)(2n-2p+3)}{16(n-p)}\neq0.
\end{eqnarray}
Remark that $\alpha\neq0$. In fact, if $\alpha=0$, then (3.59)
implies that
\begin{eqnarray*}
a_9H^9=0,
\end{eqnarray*}
which is impossible since $H$ is non-constant and $a_9\neq0$.

Put $\Phi=\frac{H}{\alpha}$. Then (3.59) reduces to a non-trivial
algebraic equation of ninth degree with respect to $\Phi$:
\begin{eqnarray}
a_9\Phi^9+a_8\Phi^8+a_7\Phi^7+a_6\Phi^6+a_5\Phi^5\
+a_4\Phi^4+a_3\Phi^3\nonumber\\
+a_2\Phi^2+a_1\Phi+a_0=0.
\end{eqnarray}
Clearly, the equation (3.61) shows that, even in case of the
existence of a real solution, $H$ is proportional to $\alpha$,
namely
\begin{eqnarray}
H=c\alpha,
\end{eqnarray}
where $c$ is a root of the equation (3.61) and has to be a nonzero
constant.

At last, we will derive a contradiction. Substituting (3.62) into
(3.37), and then applying on (3.38), (3.40) and (3.41) respectively,
we have
\begin{eqnarray}
&&e_1e_1(\alpha)-(1+\frac{2}{nc+2})\frac{e_1^2(\alpha)}{\alpha}+\frac{nc}{2}(\frac{nc}{2}+1)\alpha^3=0,\\
&&
e_1^2(\alpha)+\frac{(nc+2)\big(nc(n-p)+3nc-2(p-1)\big)}{4(n-p)}\alpha^4=0,\\
&&-e_1e_1(\alpha)+\Big[\frac{2(p-1)}{nc+2}+\frac{(n-p)\big(3nc-2(p-1)\big)}{nc(n-p)+3nc-2(p-1)}\Big]\frac{e_1^2(\alpha)}{\alpha}\nonumber\\
&&+\Big[\frac{n^2c^2}{4}+(p-1)^2+\frac{(3nc-2(p-1))^2}{4(n-p)}\Big]\alpha^3=0.
\end{eqnarray}
Substituting (3.64) into (3.63), we get
\begin{eqnarray}
e_1e_1(\alpha)+\frac{(nc+6)\big(nc(n-p)+3nc-2(p-1)\big)}{4(n-p)}\alpha^3=0.
\end{eqnarray}
Since $e_1(\alpha)\neq0$, differentiating (3.64) along $e_1$ we
obtain
\begin{eqnarray}
e_1e_1(\alpha)+\frac{2(nc+2)\big[nc(n-p)+3nc-2(p-1)\big]}{4(n-p)}\alpha^3=0.
\end{eqnarray}
Combining (3.67) with (3.66) gives
\begin{eqnarray}
\frac{(nc-2)\big[nc(n-p)+3nc-2(p-1)\big]}{4(n-p)}\alpha^3=0.
\end{eqnarray}
Since $\alpha\neq0$, we have either $nc=2$ or
$nc(n-p)+3nc-2(p-1)=0$.

In the former case, substituting $nc=2$ into (3.64) and (3.66), and
then substituting (3.64) and (3.66) into (3.65) we have
\begin{eqnarray*}
(p-1)^2+(p-1)-1+\frac{2(p-1)^2-13(p-1)+21}{n-p}=0,
\end{eqnarray*}
which reduces to
\begin{eqnarray}
(p-1)^2+(p-1)-1+\frac{2[(p-1)-\frac{13}{4}]^2-\frac{1}{8}}{n-p}=0.
\end{eqnarray}
In this case, since $p<n$, $p\geq2$ and $p\in\mathbb Z^+$, we have
always
\begin{eqnarray*}
(p-1)^2+(p-1)-1>0,
\end{eqnarray*}
and
\begin{eqnarray}
 2[(p-1)-\frac{13}{4}]^2-\frac{1}{8}\geq0.
\end{eqnarray}
Note that the equality in (3.70) holds if and only if $p=4$. the
above information shows that (3.69) gives a contradiction.

In the latter case, substituting $nc(n-p)+3nc-2(p-1)=0$ into (3.64)
and (3.66), respectively, we obtain
\begin{eqnarray*}
e_1e_1(\alpha)=e_1^2(\alpha)=0,
\end{eqnarray*}
which together with (3.65) yields a contradiction as well.

Consequently, we conclude that the mean curvature $H$ must be
constant. Therefore, biharmonic hypersurfaces with three distinct
principal curvatures in $\mathbb E^{n+1}$ have to be minimal.

In conclusion, we can state the main theorem in the following.
\begin{theorem}
Every biharmonic hypersurface with three distinct principal
curvatures in a Euclidean space with arbitrary dimension is minimal.
\end{theorem}
\begin{remark}
Remark that the approach in this paper is self-contained from a
structural point of view, and it maybe provide better insight into
the structure of biharmonic hypersurface. With this method, one
could consider hypersurfaces with four distinct principal curvatures
or the higher codimension cases of Chen's conjecture.
\end{remark}
Finally, we give an application of the main theorem.
\begin{corollary}
Every biharmonic $O(p)\times O(q)$-invariant hypersurface in a
Euclidean space $\mathbb E^{p+q}$ is minimal.
\end{corollary}
$O(p)\times O(q)$-invariant hypersurfaces, that is, invariant under
the action of some isometry group $O(p)\times O(q)$, were studied in
\cite{alencar2005}. For an $O(p)\times O(q)$-invariant hypersurface
$M$ in Euclidean space $\mathbb E^{p+q}$, it can be parameterized by
\begin{eqnarray*}
&&{\bar x}\big(t, \phi_1,\ldots, \phi_{p-1}, \psi_1, \ldots,
\psi_{q-1}\big)\\
&&=\big(x(t)\Phi(\phi_1,\ldots, \phi_{p-1}), y(t)\Psi(\psi_1,
\ldots, \psi_{q-1})\big),
\end{eqnarray*}
where $\Phi$ and $\Psi$ are orthogonal parameterizations of a unit
sphere of the corresponding dimension. It is easy to check that $M$
has at most three distinct principal curvatures, see details in
\cite{alencar2005}. Hence, by applying Theorem 3.4, we immediately
obtain Corollary 3.6.

\section*{Acknowledgments} The author would like to express his sincere gratitude to
Prof. Ye-Lin Ou for his constructive suggestions, comments and
pointing out an application (Corollary 3.6) of the main result
concerning this paper. The author is supported by the Mathematical
Tianyuan Youth Fund of China (No. 11326068), the Natural Science
Foundation of China (No. 71271045), Excellent Innovation talents
Project of DUFE (No. DUFE2014R26), and NSF of Fujian Province of
China (No. 2011J05001).


\bigskip
\address{
School of Mathematics and
Quantitative Economics\\
Dongbei University of Finance and Economics\\
Dalian 116025, P. R. China} {yufudufe@gmail.com}
\end{document}